\documentclass[12pt]{amsart}
\usepackage{amsfonts,latexsym,amsthm,amssymb,graphicx}
\usepackage{wrapfig}
\usepackage[all]{xy}
\usepackage{hyperref}

\usepackage[usenames]{color} 

\DeclareFontFamily{OT1}{rsfs}{}
\DeclareFontShape{OT1}{rsfs}{n}{it}{<-> rsfs10}{}
\DeclareMathAlphabet{\mathscr}{OT1}{rsfs}{n}{it}
\CompileMatrices

\newtheorem{theorem}{Theorem}[section]
\newtheorem{lemma}[theorem]{Lemma}
\newtheorem{corol}[theorem]{Corollary}
\newtheorem{prop}[theorem]{Proposition}

{\theoremstyle{definition} \newtheorem{defin}[theorem]{Definition}}
{\theoremstyle{remark} \newtheorem{remark}[theorem]{Remark}
\newtheorem{example}[theorem]{Example}}

\newcommand{\Cbb}{{\mathbb{C}}}
\newcommand{\Pbb}{{\mathbb{P}}}
\newcommand{\one}{1\hskip-3.5pt1}
\newcommand{\csm}{{c_{\text{SM}}}}
\DeclareMathOperator{\coker}{coker}
\DeclareMathOperator{\im}{im}
\newcommand{\qede}{\hfill$\lrcorner$}
\DeclareMathOperator{\Der}{Der}
\DeclareMathOperator{\Ext}{Ext}
\DeclareMathOperator{\Hom}{Hom}
\DeclareMathOperator{\Spla}{Splay}
\newcommand{\cI}{{\mathscr I}}
\newcommand{\cO}{{\mathscr O}}

\title{
Splayed divisors and their Chern classes
}
\author{Paolo Aluffi}
\author{Eleonore Faber}
\address{
Mathematics Department, 
Florida State University,
Tallahassee FL 32306, U.S.A.
}
\email{aluffi@math.fsu.edu}
\address{
Institut f\"ur Mathematik,
Nordbergstra\ss e 15,
1090 Wien,
\"Osterreich
}
\email{eleonore.faber@univie.ac.at}

\begin{document}

\begin{abstract}
We obtain several new characterizations of splayedness for divisors: a Leibniz
property for ideals of singularity subschemes, the vanishing of a `splayedness'
module, and the requirements that certain natural morphisms of modules and
sheaves of logarithmic derivations and logarithmic differentials be isomorphisms.
We also consider the effect of splayedness on the Chern classes of sheaves of 
differential forms with logarithmic poles along splayed divisors, as well as on the 
Chern-Schwartz-MacPherson classes of the complements of these divisors.
A postulated relation between these different notions of Chern class leads
to a conjectural identity for Chern-Schwartz-MacPherson classes of splayed
divisors and subvarieties, which we are able to verify in several template situations.
\end{abstract}

\maketitle


\section{Introduction}\label{intro}
Two divisors in a nonsingular variety $V$ are {\em splayed\/} at a point $p$ if their
local equations at $p$
may be written in terms of disjoint sets of analytic coordinates.
Splayed divisors are transversal in a very strong sense; indeed, splayedness
may be considered a natural generalization of transversality for possibly singular
divisors (and subvarieties of higher codimension). In previous work, the 
second-named author has obtained several characterizing properties for 
splayedness (\cite{Faber12}). For example, two divisors are splayed at $p$ if 
and only if their {\em Jacobian ideals\/} satisfy a `Leibniz property' (\cite{Faber12}, 
Proposition~8); and if and only if the corresponding modules of logarithmic 
derivations at $p$ span the module of ordinary derivations for $V$ at $p$
(\cite{Faber12}, Proposition~15).

In this paper we refine some of these earlier results, and consider implications
for different notions of {\em Chern classes\/} associated with divisors.
Specifically, we strengthen the first result recalled above, by showing that
splayedness is already characterized by the Leibniz property after restriction to
the union of the divisors; equivalently, this amounts to a Leibniz property for
the ideals defining the {\em singularity subschemes\/} for the divisors at $p$
(Corollary~\ref{corol:Jpsplayed} and~\ref{corol:Jsplayed}).
We introduce a `splayedness module', which we describe both in terms of 
these ideals and in terms of modules of logarithmic derivations
(Definition~\ref{Def:splayedmod}, Proposition~\ref{prop:splayedmod}), and whose 
vanishing is equivalent to splayedness. Thus, this module quantifies precisely
the failure of splayedness of two divisors meeting at a point.

These results may be expressed in terms of the quality of certain 
natural morphisms associated with two divisors. For example, given two 
divisors $D_1$, $D_2$ meeting at a point $p$ 
and without common components,
there is a natural monomorphism
\[
\frac{\Der_{V,p}}{\Der_{V,p}(-\log(D_1\cup D_2))} \hookrightarrow
\frac{\Der_{V,p}}{\Der_{V,p}(-\log D_1)}\oplus
\frac{\Der_{V,p}}{\Der_{V,p}(-\log D_2)}
\]
involving quotients of modules of logarithmic derivations.
We prove that $D_1$ and $D_2$ are splayed at $p$ if and only if this monomorphism
is an isomorphism (Theorem~\ref{Thm:Dermodcrit}). We also prove an analogous
statement involving sheaves of {\em logarithmic differentials\/} 
(Theorem~\ref{Thm:Omegamodcrit}) giving a partial 
answer to a question raised in~\cite{Faber12}, but only subject to the vanishing of
an $\Ext$ module: $D_1$ and~$D_2$ are splayed if the natural inclusion
\begin{equation}\label{eq:natinclusion}
\Omega^1_{V,p}(\log D_1) + \Omega^1_{V,p}(\log D_2)\subseteq \Omega^1_{V,p}(\log D)
\end{equation}
is an equality and $\Ext^1_\cO(\Omega^1_{V,p}(\log D),\cO) =0$.
Thus, if $D$ is free at $p$, then $D_1$ and~$D_2$ are splayed at $p$ if and only if
the two modules in \eqref{eq:natinclusion} are equal. In general this condition 
alone does not imply splayedness, as Example~\ref{Ex:additivomega} shows.

One advantage of expressing splayedness in terms of these morphisms is that the
characterizing conditions globalize nicely, and give conditions on morphisms of
{\em sheaves\/} of logarithmic derivations and differentials characterizing splayedness
at all points of intersection of two divisors. These conditions imply identities involving
Chern classes for these sheaves (Corollary~\ref{cor:Chern}). In certain situations 
(for example in the case
of curves on surfaces) these identities actually characterize splayedness. 
Also, there is a different notion of `Chern class' that can be associated with a 
divisor~$D$ in a nonsingular variety $V$, namely the Chern-Schwartz-MacPherson 
($\csm$) class of the 
complement $V\smallsetminus D$. (See~\S\ref{CSMintro} for a rapid reminder of
this notion). In previous work, the first-named author has determined several situations
where this $\csm$ class {\em equals\/} the Chern class $c(\Der_V(-\log D))$ of 
the sheaf of logarithmic differentials. It is then natural to expect that $\csm$
classes of complements of splayed divisors, and more general subvarieties,
should satisfy a similar type of relations as the one obtained for ordinary
Chern classes of sheaves of derivations. From this point of view we analyze 
three template sources
of splayed subvarieties: subvarieties defined by pull-backs from factors of a 
product (Proposition~\ref{pro:products}), joins of projective varieties 
(Proposition~\ref{pro:joins}), and the case of curves (Proposition~\ref{pro:curves}). 
In each of the three  
situations we are able to verify explicitly that the corresponding expected relation
of $\csm$ classes does hold. We hope to come back to the question of the validity
of this relation for arbitrary splayed subvarieties in future work.

The new characterizations for splayedness are given in \S\ref{sec:splay},
together with the implications for Chern classes of sheaves of logarithmic
derivations for splayed divisors. The conjectured expected relation for
$\csm$ classes of complements, together with some necessary background
material, is presented in~\S\ref{sec:CSMsplayed}.
\smallskip

{\em Acknowledgements:\/} 
We thank J\"org Sch\"urmann and Mathias Schulze for comments on a previous
version of this paper.
P.~A.~thanks Caltech for hospitality during the preparation of this note. 
E.~F.~has been supported by a For Women in Science award 2011 of 
L'Or{\'e}al Austria, the Austrian commission for UNESCO and the Austrian Academy 
of Sciences and by the Austrian Science Fund (FWF) in frame of project P21461.


\section{Splayedness}\label{sec:splay}

\subsection{}
Let $V$ be a smooth complex projective variety of dimension~$n$.
We say that two divisors $D_1$ and $D_2$ in $V$ are {\em splayed\/} at a point 
$p$ if there exist complex analytic coordinates $x_1,\ldots,x_n$ at $p$ such
that $D_1$, $D_2$ have defining equations $g(x_1, \ldots, x_k,0, \ldots, 0)=0$, 
$h(0, \ldots, 0, x_{k+1}, \ldots, x_n)=0$
 at $p$, 
where $1 \leq k < n$. Here $g,h \in \cO^{\text{an}}_{V,p}\cong 
{\mathbb C}\{x_1, \ldots, x_n\}$. 
We simply say that $D_1$ and $D_2$ are {\em splayed\/} if they are splayed
at $p$ for every $p\in D_1\cap D_2$.

The {\em Jacobian ideal\/} $J_f$ of $f\in \mathbb C\{x_1,\ldots , x_n\}$ is the ideal 
generated by its partial derivatives, i.e., 
$J_f=(\partial_{x_1}f, \ldots, \partial_{x_n}f)$. 
We will also consider the ideal $J'_f=J_f+(f)$. 
A function $f$ is called {\em Euler-homogeneous at $p$\/} if $J'_f=J_f$, that is, if
there exists a derivation $\delta$ such that $f=\delta f$.

Unlike the Jacobian ideal, the ideal $J'_f$ only depends on the associate class of
$f$: if $u$ is a unit, then $J'_{uf}=J'_f$. This implies that the
ideal $J'_f$ globalizes, in the sense that if
$D$ is a divisor with local equation $f_p=0$ at $p$, the ideals 
$J'_{f_p}\subseteq \cO^{\text{an}}_{V,p}$ for $p\in V$ determine 
a subscheme $JD$ of $D$, which we call the {\em singularity subscheme\/}
of $D$.

In other words, the ideal sheaf of $JD$ in $D$ is the image of the natural morphism of 
sheaves
\[
\xymatrix{
\Der_V(-D) \ar[r] & \cO_D
}
\]
defined by applying derivations to local equations for $D$. The kernel of the
corresponding morphism $\Der_V \to \cO_D(D)$ defines the sheaf of
{\em logarithmic derivations\/}
(or {\em logarithmic vector fields\/}) with respect to $D$.
We briefly recall the notions of logarithmic derivations and differential forms and free 
divisors, following \cite{Saito80}. Let $D \subseteq V$ be a divisor that is locally at $p$ given 
by $\{f=0\}$. A derivation $\delta \in \Der_{V,p}$ at $p$ is {\em logarithmic along $D$\/}
if the germ $\delta(f)$ is contained in the ideal $(f)$ of $\cO_{V,p}$.
The module of germs of logarithmic derivations (along $D$) at $p$ is denoted by 
\[  
\Der_{V,p}(-\log D)=\{  \delta: \delta \in \Der_{V,p} \text{ such that }\delta f \in (f)\cO_{V,p}  \}. 
\]
These modules are the stalks at points $p$ of the sheaf $\Der_V(-\log D)$ of 
$\cO_{V}$-modules. Similarly we define logarithmic differential forms: a meromorphic 
$q$-form $\omega$ is \emph{logarithmic} (along $D$) at $p$ if $\omega f$ and 
$fd\omega$ are holomorphic (or equivalently if $\omega f$ and $df \wedge \omega$ 
are holomorphic) at $p$.
 We denote
\[\Omega^q_{V,p}(\log D)= \{ \omega: \omega \text{ germ of a logarithmic $q$-form at $p$} \}.\] 
Again, this notion globalizes and yields a coherent sheaf $\Omega^q_{V}(\log D)$ of $\cO_V$-modules.
One can show that $\Der_{V,p}(-\log D)$ and $\Omega^1_{V,p}(\log D)$ are reflexive 
$\cO_{V,p}$-modules dual to each other
(see \cite{Saito80}, Corollary~1.7).  
The germ $(D,p)$ is called \emph{free} if
$\Der_{V,p}(-\log D)$ resp. $\Omega_{V,p}^1(\log D)$ is a free $\cO_{V,p}$-module. The divisor $D$ is called a \emph{free divisor} if $(D,p)$ is free at every point $p \in V$.

In terms of $\cO=\cO^{\text{an}}_{V,p}$-modules of 
derivations, if $D$ has equation $f=0$ at $p$, then there is an exact 
sequence of $\cO$-modules
\begin{equation}\label{seq:localder}
\xymatrix{
0 \ar[r] & \Der_{V,p}(-\log D) \ar[r] & \Der_{V,p} \ar[r] & J'_f/(f)
\ar[r] & 0\quad.
}
\end{equation}
This sequence is the local analytic aspect of the sequence of coherent 
$\cO_V$-modules
\begin{equation}\label{seq:globalder}
\xymatrix{
0 \ar[r] & \Der_V(-\log D) \ar[r] & \Der_V \ar[r] & \cI_{JD,D}(D)
\ar[r] & 0\quad.
}
\end{equation}
where $\cI_{JD,D}$ is the ideal sheaf of $JD$ in $D$ (see e.g., 
\cite{MR2359100}, \S2).

\subsection{}
Equivalent conditions for splayedness in terms of Jacobian ideals and
modules of logarithmic derivations are explored in \cite{Faber12}. 
In this section we reinterpret some of the results of \cite{Faber12},
and discuss other characterizations.

Let $\cO=\mathbb C\{x_1,\dots,x_n\}$. We will assume throughout that $D_1$
and $D_2$ are divisors defined by $g,h\in \cO$ respectively, where $g$ and $h$
are reduced and without common components. According to Proposition~8 
in~\cite{Faber12}, $D_1$ and $D_2$ are splayed if and only if 
(up to possibly multiplying $g$ and $h$ by units)
$J_{gh}$ satisfies 
the Leibniz property:
\[
J_{gh}=g J_h+h J_g\quad.
\]
We will prove that this equality is equivalent to the same property for $J'$:
\[
J'_{gh}=h J'_h+h J_g'\quad,
\]
and interpret this equality in terms of modules of derivations.

\begin{lemma}\label{Lem:injmorJp}
If $g,h\in \cO$ have no components in common, then there is an injective
homomorphism of $R$-modules
\[
\varphi\colon \frac{J'_g}{(g)} \oplus \frac{J'_h}{(h)} \hookrightarrow \frac{\cO}{(gh)}
\]
given by multiplication by $h$ on the first factor and by $g$ on the second.
The image of~$\varphi$ contains $J'_{gh}/(gh)$.
\end{lemma}

\begin{proof}
The homomorphism $\varphi$ is given by 
\[
(a + (g),b + (h)) \mapsto ah + bg \mod (gh)\quad.
\]
It is clear that $\varphi$ is well-defined. To verify that $\varphi$ is injective
under the assumption that $g,h$ do not contain a common factor, assume
$\varphi(a,b) = 0 \mod (gh)$; then $ah +bg= c gh$ for some representatives
$a,b \in \cO$ and $c \in \cO$. This implies that $ah\in (g)$, and hence $a\in (g)$
since $g$ and $h$ have no common factors. By the same
token, $b\in (h)$. Hence $a$ and $b$ have to be zero in $J_g'/(g)$ and $J_h'/(h)$.
Finally, $J'_{gh}/(gh)$ is generated by $\partial_{x_i}(gh) \mod (gh)$.
Since $\partial_{x_i}(gh)=h\partial_{x_i} g+g \partial_{x_i} h\in hJ_g+gJ_h$,
we see that $\partial_{x_i}(gh)+(gh)\in hJ'_g/(gh)+gJ'_h/(gh)=\im \varphi$,
and hence $J'_{gh}/(gh)\subseteq \im \varphi$, as claimed.
\end{proof}

The following result expresses the splayedness condition in terms of the 
morphism~$\varphi$ of Lemma~\ref{Lem:injmorJp}.

\begin{theorem}\label{Thm:Jpcrit}
Let $D_1$, $D_2$ be divisors of $V$, given by $g=0$, $h=0$ at $p$, 
where $g$ and~$h$ have no common components. Then 
$D_1$ and $D_2$ are splayed at $p$ if and only if the morphism $\varphi$ of
Lemma~\ref{Lem:injmorJp} induces an isomorphism
\[
\frac{J'_g}{(g)} \oplus \frac{J'_h}{(h)} \cong \frac{J'_{gh}}{(gh)}
\quad.\]
\end{theorem}

\begin{proof}
By Lemma~\ref{Lem:injmorJp}, there is an injective homomorphism
\[
\iota\colon \frac{J'_{gh}}{(gh)} \hookrightarrow 
\frac{J'_g}{(g)} \oplus \frac{J'_h}{(h)}\quad.
\]
This morphism $\iota$ is the unique homomorphism such that the composition
\[
\xymatrix{
\displaystyle \frac{J'_{gh}}{(gh)} \ar@{^(->}[r]^-\iota &
\displaystyle \frac{J'_g}{(g)} \oplus \frac{J'_h}{(h)} \ar@{^(->}[r]^-\varphi &
\displaystyle \frac{\cO}{(gh)}
}
\]
is the natural inclusion $r + (gh) \mapsto r + (gh)$. Explicitly, an element of
$J'_{gh}/(gh)$ is of the form $\delta(gh) + (gh)$, for a derivation $\delta$.
Since $\delta(gh)=h\delta(g) + g\delta(h)$, and
\[
h\delta(g) + g\delta(h) + (gh) = \varphi( \delta(g) + (g), \delta(h) + (h))\quad,
\]
we see that
\[
\iota( \delta(gh) + (gh) ) = (\delta(g) + (g), \delta(h) + (h))\quad.
\]
In other words, $\iota$ is the morphism induced by the compatibility
with the morphisms of $\Der$ modules: it is the unique homomorphism 
making the following diagram commute:
\[
\xymatrix{
\displaystyle \frac{J'_{gh}}{(gh)} \ar@{^(->}[r]^\iota &
\displaystyle \frac{J'_g}{(g)} \oplus \frac{J'_h}{(h)} \\
\displaystyle \frac{\Der_{V,p}}{\Der_{V,p}(-\log D)} \ar[u]_\cong \ar@{^(->}[r] &
\displaystyle \frac{\Der_{V,p}}{\Der_{V,p}(-\log D_1)} \oplus \frac{\Der_{V,p}}{\Der_{V,p}(-\log D_2)} 
\ar[u]^\cong
}
\]
where $D=D_1\cup D_2$, and the vertical isomorphisms are induced by 
sequence~\eqref{seq:localder}.
The monomorphism in the bottom row is induced from the natural morphism
\[
\xymatrix{
\Der_{V,p} \ar[r] &
\displaystyle \frac{\Der_{V,p}}{\Der_{V,p}(-\log D_1)} \oplus 
\frac{\Der_{V,p}}{\Der_{V,p}(-\log D_2)}\quad,
}
\]
using the fact that the kernel of this morphism is 
$\Der_{V,p}(-\log D_1) \cap \Der_{V,p}(-\log D_2)$, and 
by Seidenberg's theorem $\Der_{V,p}(-\log D_1) \cap \Der_{V,p}(-\log D_2) = 
\Der_{V,p}(-\log D)$, see e.g., \cite{HM93}, p.~313 or~\cite{MR1217488}, Proposition~4.8.
This morphism (and hence $\iota$) is onto if and only if $\Der_{V,p}= 
\Der_{V,p}(-\log D_1) + \Der_{V,p}(-\log D_2)$, 
and this condition is satisfied if and only if $D_1$ and $D_2$ are splayed at $p$ by 
Proposition~15 of \cite{Faber12}. The statement follows.
\end{proof}

\subsection{}
The argument given above may be recast as follows. As $\Der_{V,p}(-\log D_1)$
and $\Der_{V,p}(-\log D_2)$ are submodules of $\Der_{V,p}$, whose intersection is
$\Der_{V,p}(D)$, we have an exact sequence
\begin{multline*}
0 \longrightarrow \frac{\Der_{V,p}}{\Der_{V,p}(-\log D)} \longrightarrow
\frac{\Der_{V,p}}{\Der_{V,p}(-\log D_1)} \oplus \frac {\Der_{V,p}}{\Der_{V,p}(-\log D_2)} \\
\longrightarrow \frac{\Der_{V,p}}{\Der_{V,p}(-\log D_1)+\Der_{V,p}(-\log D_2)}
\longrightarrow 0
\end{multline*}
The last quotient may be viewed as a measure of the `failure of splayedness at $p$'.

\begin{defin}\label{Def:splayedmod}
The {\em splayedness module\/} for $D_1$ and $D_2$ at $p$ is the quotient
\[
\Spla_p(D_1,D_2):=\frac {\Der_{V,p}}{\Der_{V,p}(-\log D_1)+\Der_{V,p}(-\log D_2)}\quad.
\]
\end{defin}
By Proposition~15 in~\cite{Faber12}, $D_1$ and $D_2$ are splayed at $p$
if and only if their splayedness module at $p$ vanishes. Equivalently:

\begin{theorem}\label{Thm:Dermodcrit}
Let $D_1$, $D_2$ be reduced divisors of $V$, without common components, and
let $D=D_1\cup D_2$. Then there is a natural monomorphism of modules
\[
\xymatrix{
\displaystyle
\frac{\Der_{V,p}}{\Der_{V,p}(-\log D)} \ar@{^(->}[r] &
\displaystyle
\frac{\Der_{V,p}}{\Der_{V,p}(-\log D_1)}\oplus \frac{\Der_{V,p}}{\Der_{V,p}(-\log D_2)}\quad,
}
\]
and $D_1$, $D_2$ are splayed at $p$ if and only if this monomorphism is an 
isomorphism.
\end{theorem}

The splayedness module may be computed as follows:

\begin{prop}\label{prop:splayedmod}
With notation as above, the splayedness module is isomorphic to
\[
\frac {hJ'_g+gJ'_h}{J'_{gh}} \quad.
\]
\end{prop}

\begin{proof}
Via the identification used in the proof of Theorem~\ref{Thm:Jpcrit}, the monomorphism
appearing in Theorem~\ref{Thm:Dermodcrit} is the homomorphism
\[
\iota\colon \frac{J'_{gh}}{(gh)} \hookrightarrow 
\frac{J'_g}{(g)} \oplus \frac{J'_h}{(h)}\quad.
\]
Therefore, the cokernels are isomorphic; this shows that $\coker\iota$ is isomorphic
to the splayedness module. To determine $\coker\iota$, use the monomorphism
$\varphi$ of Lemma~\ref{Lem:injmorJp} to identify the direct sum with a submodule
of $\cO/(gh)$; this submodule is immediately seen to equal $(hJ'_g+gJ'_h)/(gh)$.
Use this identification to view $\iota$ as acting
\[
\frac{J'_{gh}}{(gh)} \hookrightarrow \frac{hJ'_g+gJ'_h}{(gh)}\quad;
\]
it is then clear that $\coker\iota$ is isomorphic to the module given in the statement.
\end{proof}

\begin{corol}\label{corol:Jpsplayed}
With notation as above, $D_1$ and $D_2$ are splayed at $p$ if and only if
$J'_{gh} = hJ'_g + g J'_h$.
\end{corol}

\begin{corol}\label{corol:Jsplayed}
With notation as above, $D_1$ and $D_2$ are splayed at $p$ if and only if
$J_{gh}+(gh)=hJ_g + gJ_h+(gh)$.
\end{corol}

\begin{proof}
This is a restatement of Corollary~\ref{corol:Jpsplayed}.
\end{proof}

\subsection{}\label{Ssec:JvsJp}
As recalled above, Proposition~8 from~\cite{Faber12} states that $D_1$ and $D_2$ 
are splayed at~$p$ if and only if $J_{gh}=hJ_g + gJ_h$
up to multiplying $g$ and $h$ by units.
Corollary~\ref{corol:Jsplayed} 
strengthens this result, as it shows that the weaker condition that these two ideals 
are equal {\em modulo $(gh)$\/} suffices to imply splayedness.
In fact, this gives an alternative proof of Proposition~8 from~\cite{Faber12}: as
\begin{align*}
\text{$D_1$ and $D_2$ are splayed at $p$} & \implies 
\text{$J_{gh} = h J_g + gJ_h$ 
(for suitable choices of $g$ and $h$)
}\\
& \implies J_{gh} = h J_g + g J_h \mod (gh) \\
& \implies \text{$D_1$ and $D_2$ are splayed at $p$}
\end{align*}
(the first two implications are immediate, and the third is given by 
Corollary~\ref{corol:Jsplayed}), these conditions are all equivalent.
Also, note that the conditions expressed in Corollaries~\ref{corol:Jpsplayed}
and~\ref{corol:Jsplayed} are insensitive to multiplications by units. Indeed, if
$f\in\cO$ and $u$ is a unit, then $J'_{fu}=J'_f$. In general, $J_{fu}\ne J_f$.

\begin{remark}\label{rem:Eulhom}
The implication 
\[
J_{gh} = hJ_g + gJ_h \mod (gh) \implies J_{gh} = hJ_g + gJ_h
\]
is straightforward if $gh$ is Euler homogeneous at $p$, and then it does not 
require a particular choice of $g$, $h$ defining $D_1$, $D_2$. Indeed, the
inclusion $J_{gh} \subseteq hJ_g + gJ_h$ always holds; to verify the other
inclusion, let $\delta_1$, $\delta_2$ be derivations, and consider
$h \delta_1 g + g \delta_2 h$.
By the equality $J_{gh} = hJ_g + gJ_h \mod (gh)$, there exists a derivation $\delta$
and an element $a$ such that
\[
h \delta_1 g + g \delta_2 h = \delta(gh) + agh \ .
\]
If $gh$ is Euler homogeneous, we can find a derivation $\varepsilon$ such that
$gh=\varepsilon(gh)$; thus
\[
h \delta_1 g + g \delta_2 h = (\delta+a \varepsilon) (gh)\ ,
\]
and this shows $hJ_g + gJ_h\subseteq J_{gh}$ as $\delta_1$ and $\delta_2$
were arbitrary.
\qede\end{remark}

\begin{remark}
In view of Proposition~8 from \cite{Faber12}, it is natural to ask whether the module
$(hJ_g+g J_h)/J_{gh}$ may be another realization of the splayedness module.
This is not the case. Examples may be constructed by considering Euler-homogeneous
functions $g$, $h$ (i.e., assume $J'_g=J_g$, $J'_h=J_h$) such that the product is not 
Euler-homogeneous: 
concretely, one
may take $g=x^3+y^2$ and $h=x^5+y^7$.
For such functions, the splayedness module is
\[
\frac{hJ'_g+gJ'_h}{J'_{gh}}=\frac{hJ_g+gJ_h}{J'_{gh}}\quad;
\]
$(hJ_g+gJ_h)/J_{gh}$ surjects onto this module, but not isomorphically since the
kernel $J'_{gh}/J_{gh}$ is nonzero if $gh$ is not Euler-homogeneous.

On the other hand, $(hJ_g+gJ_h)/J_{gh}$ {\em does\/} equal the splayedness module
if $gh$ is Euler homogeneous. Indeed, the equality $J_{gh}=J'_{gh}$ implies that
$hJ'_g + gJ'_h \subseteq h J_g+gJ_h$ by the same argument used in 
Remark~\ref{rem:Eulhom}, and the other inclusion is always true.
\qede\end{remark}

\subsection{}
It is natural to ask whether splayedness can be expressed in terms of
logarithmic {\em differential forms,\/} in the style of 
Theorem~\ref{Thm:Dermodcrit} (cf.~Question~16 in~\cite{Faber12}).
A partial answer to this question will be given in 
Theorem~\ref{Thm:Omegamodcrit} below.
We first give a precise (but a little obscure) 
translation of splayedness in terms of a morphism of $\Ext$ modules
of modules of differential forms. We consider the natural epimorphism
\[
\Omega^1_{V,p}(\log D_1)\oplus \Omega^1_{V,p}(\log D_2) \longrightarrow
\Omega^1_{V,p}(\log D_1) + \Omega^1_{V,p}(\log D_2) 
\]
and the induced morphism
\begin{equation}\label{eq:Extmorph}
\Ext^1_{\cO}(\Omega^1_V(\log D_1) + \Omega^1_V(\log D_2),\cO) \longrightarrow
\Ext^1_{\cO}(\Omega^1_V(\log D_1)\oplus \Omega^1_V(\log D_2),\cO)\quad.
\end{equation}

\begin{prop}\label{prop:Extcond}
Let $D_1$, $D_2$ be reduced divisors of $V$, without common components, and
let $D=D_1\cup D_2$. Then the morphism \eqref{eq:Extmorph} is an epimorphism,
and its kernel is the splayedness module: there is an exact sequence
\begin{multline*}
0 \longrightarrow
\Spla_p(D_1,D_2) \longrightarrow
\Ext^1_{\cO}(\Omega^1_V(\log D_1) + \Omega^1_V(\log D_2),\cO) \\ \longrightarrow
\Ext^1_{\cO}(\Omega^1_V(\log D_1)\oplus \Omega^1_V(\log D_2),\cO) \longrightarrow 0
\end{multline*}
\end{prop}

Therefore, $D_1$ and $D_2$ are splayed at $p$ if and only if the morphism
\eqref{eq:Extmorph} is an isomorphism, if and only if it is injective. 

\begin{proof}
Note that $\Omega^1_{V,p}(\log D_1)$, $\Omega^1_{V,p}(\log D_2)$ both
embed in $\Omega^1_{V,p}(\log D)$, and an elementary verification shows that
\[
\Omega^1_{V,p}(\log D_1) \cap \Omega^1_{V,p}(\log D_2)
=\Omega^1_{V,p}
\]
if $D_1$ and $D_2$ have no components in common. Indeed, the poles of a
form in the intersection would be on a divisor contained in both $D_1$ and 
$D_2$. By the previous assumption on $D_1$ and $D_2$ such a form can 
have no poles. We then get an exact sequence
\[
\xymatrix@C=20pt{
0 \ar[r] & \Omega^1_{V,p} \ar[r] &
\Omega^1_{V,p}(\log D_1) \oplus \Omega^1_{V,p}(\log D_2) \ar[r] &
\Omega^1_{V,p}(\log D_1) + \Omega^1_{V,p}(\log D_2) \ar[r] &
0
}
\]
Applying the dualization functor $\Hom_\cO(-,\cO)$ to this sequence gives 
the exact sequence
\begin{align*}
0 & \rightarrow 
(\Omega^1_{V,p}(\log D_1)+\Omega^1_{V,p}(\log D_2))^\vee \rightarrow
\Der_{V,p}(-\log D_1) \oplus \Der_{V,p}(-\log D_2) \rightarrow \Der_{V,p} \\
& \rightarrow
\Ext^1_{\cO}(\Omega^1_V(\log D_1) + \Omega^1_V(\log D_2),\cO) \rightarrow
\Ext^1_{\cO}(\Omega^1_V(\log D_1)\oplus \Omega^1_V(\log D_2),\cO) 
\rightarrow 0
\end{align*}
(the last $0$ is due to the fact that $\Der_{V,p}$ is free by the hypothesis of
nonsingularity of~$V$). This shows that the morphism \eqref{eq:Extmorph}
is an epimorphism, and identifies its kernel with the cokernel of the natural
morphism $\Der_{V,p}(-\log D_1) \oplus \Der_{V,p}(-\log D_2) \rightarrow 
\Der_{V,p}$, which is the splayedness module introduced in 
Definition~\ref{Def:splayedmod}.
\end{proof}

\begin{remark}\label{Rem:dualplus}
The argument also shows that 
$(\Omega^1_{V,p}(\log D_1)+\Omega^1_{V,p}(\log D_2))^\vee$
is isomorphic to $\Der_{V,p}(-\log D)$ regardless of splayedness. Indeed,
by the sequence obtained in the proof this dual is identified with
$\Der_{V,p}(-\log D_1) \cap \Der_{V,p}(-\log D_2)$, and this equals
$\Der_{V,p}(-\log D)$ by Seidenberg's theorem as recalled in the proof
of Theorem~\ref{Thm:Jpcrit}.
\qede\end{remark}

While the statement of Proposition~\ref{prop:Extcond} is precise, it seems
hard to apply. The following result translates this criterion in terms that
are more similar to those of Theorem~\ref{Thm:Dermodcrit}, but at the 
price of a hypothesis of freeness.

\begin{theorem}\label{Thm:Omegamodcrit}
Let $D_1$, $D_2$ be reduced divisors of $V$, without common components, and
let $D=D_1\cup D_2$. Then there is a natural monomorphism of modules
\[
\xymatrix{
\displaystyle
\frac{\Omega^1_{V,p}(\log D_1)\oplus \Omega^1_{V,p}(\log D_2)}{\Omega^1_{V,p}} 
\ar@{^(->}[r] &
\Omega^1_{V,p}(\log D)\quad.
}
\]
If $D_1$, $D_2$ are splayed at $p$, then this monomorphism is an isomorphism.

If $\Ext^1_{\cO}(\Omega^1_{V,p}(\log D),\cO)=0$ (for example, if $D$ is free at $p$), 
then the converse implication holds.
\end{theorem}

\begin{remark}\label{Rem:aboutOme}
The condition in the statement is clearly equivalent to the condition that the
inclusion $\Omega^1_{V,p}(\log D_1) + \Omega^1_{V,p}(\log D_2)
\subseteq \Omega^1_{V,p}(\log D)$ be an equality. 
Example~\ref{Ex:additivomega} below shows that there are non-splayed
divisors for which this condition does hold. Thus the situation for logarithmic
{\em differentials\/} {\em vis-a-vis\/} splayedness appears to be less straightforward 
than for logarithmic {\em derivations.\/} Also see~Remark~\ref{rem:additivomega}.
\qede\end{remark}

\begin{proof}
To show that splayedness implies the stated condition, assume that $D_1$ 
and~$D_2$ are splayed at $p$ and defined by $g=g(z_1, \ldots, z_p)$ and 
$h=h(z_{p+1}, \ldots, z_n)$. Since $g$ and $h$ are reduced,
we may assume that there exist indices $i$, resp., $j$
such that $g$ and $\partial_{x_i}g$, resp., $h$ and $\partial_{x_j}h$ have
no common factors.
By definition, a meromorphic differential one-form has logarithmic poles along 
$D=D_1 \cup D_2$ if $gh \omega$ and $d(gh) \wedge \omega$ are holomorphic. Writing 
$$\omega= \frac{\sum_{i=1}^p a_i dz_i + \sum_{j=p+1}^n b_j dz_j}{gh},$$ 
the second condition yields that each $a_i$ divisible by $h$, i.e., $a_i=h \tilde a_i$ 
(and similarly
each $b_j=g \tilde b_j$) for some $\tilde a_i, \tilde b_j \in \cO$. Hence $\omega$ is of the form 
$\omega_g+\omega_h=\frac{\sum_{i=1}^p \tilde a_i dz_i}{g}+\frac{\sum_{j=p+1}^{n} 
\tilde b_j dz_j}{h}$. It is easy to see that $\omega_g \in \Omega^1_{V,p}(\log D_1)$ 
and $\omega_h \in \Omega^1_{V,p}(\log D_2)$.

To prove that the stated condition implies splayedness if 
$\Ext^1_\cO(\Omega^1_{V,p}(\log D),\cO)=0$, we use the
exact sequence in Proposition~\ref{prop:Extcond}.
If $\Omega^1_{V,p}(\log D_1) + \Omega^1_{V,p}(\log D_2) = \Omega^1_{V,p}(\log D)$,
then $\Ext^1_\cO(\Omega^1_{V,p}(\log D_1) + \Omega^1_{V,p}(\log D_2),\cO) 
= \Ext^1_\cO(\Omega^1_{V,p}(\log D),\cO)=0$. In this case the exact sequence in 
Proposition~\ref{prop:Extcond} becomes
\[
\xymatrix{
0 \ar[r] &
\Spla_p(D_1,D_2) \ar[r] & 
0 \ar[r] &
\Ext^1_{\cO}(\Omega^1_V(\log D_1)\oplus \Omega^1_V(\log D_2),\cO)
\ar[r] & 0
}
\]
and forces the splayedness module to vanish, concluding the proof.
\end{proof}

\begin{corol}
If $C_1$, $C_2$ are curves without common components on a nonsingular
surface $S$, then $C_1$ and $C_2$ are splayed at $p$ if and only if
\[
\Omega_{S,p}(\log C_1)+\Omega_{S,p}(\log C_2) = \Omega_{S,p}(\log (C_1\cup C_2))\quad.
\]
\end{corol}

\begin{proof}
Indeed, the additional condition $\Ext^1_\cO(\Omega_{S,p}(\log(C_1\cup C_2),\cO)=0$ 
is automatic in this case, since the locus along which a reflexive sheaf on a nonsingular 
variety is not free has codimension at least $3$ (\cite{MR597077}, Corollary~1.4)
and sheaves of logarithmic differentials are reflexive.
\end{proof}

\begin{example}  \label{Ex:additivomega}
Let $D=D_1 \cup D_2$ be the union of the cone $D_1=\{ h_1=x^2+y^2-z^2=0\}$ and 
the plane $D_2=\{ h_2=x=0\}$ in $(V,p)=(\Cbb^3,0)$. Then $D$ is neither splayed  
nor free at the origin. Indeed, $\Der_{V,p}(-\log D_1)$ is generated by $x\partial_x + 
y \partial_y+z\partial_z, y\partial_x -x \partial_y, z \partial_x + x \partial_z, z\partial_y 
+ y \partial_z$ and $\Der_{V,p}(-\log D_2)$ by $x \partial_x, \partial_y$, $\partial_z$; 
it follows that $\partial_x$ is not contained in $\Der_{V,p}(-\log D_1)
+\Der_{V,p}(-\log D_2)$. Thus $\Der_{V,p}(-\log D_1)+\Der_{V,p}(-\log D_2)
\ne \Der_{V,p}(-\log D)$, and hence $D$ is not splayed, by Proposition~15 
of~\cite{Faber12} (i.e., $\Spla_p(D_1,D_2)\ne 0$). 
On the other hand, it is easy to see that 
\[
\Omega^1_{V,p}(\log D_1)+\Omega_{V,p}^1(\log D_2)=\frac{dh_1}{h_1}\cO 
+ \frac{dx}{x}\cO + dy\cO +dz\cO.
\] 
By Theorem 2.9 of \cite{Saito80}, 
$\Omega_{V,p}^1(\log D)=\Omega_{V,p}^1(\log D_1)+\Omega_{V,p}^1(\log D_2)$. 
Note that it follows that $D$ cannot be free at $p$, by Theorem~\ref{Thm:Omegamodcrit}.
A computation shows that $\Der_{V,p}(-\log D)$ is minimally generated by 4 derivations,
confirming this.
\begin{figure}[!h]
\includegraphics[width=0.5 \textwidth]{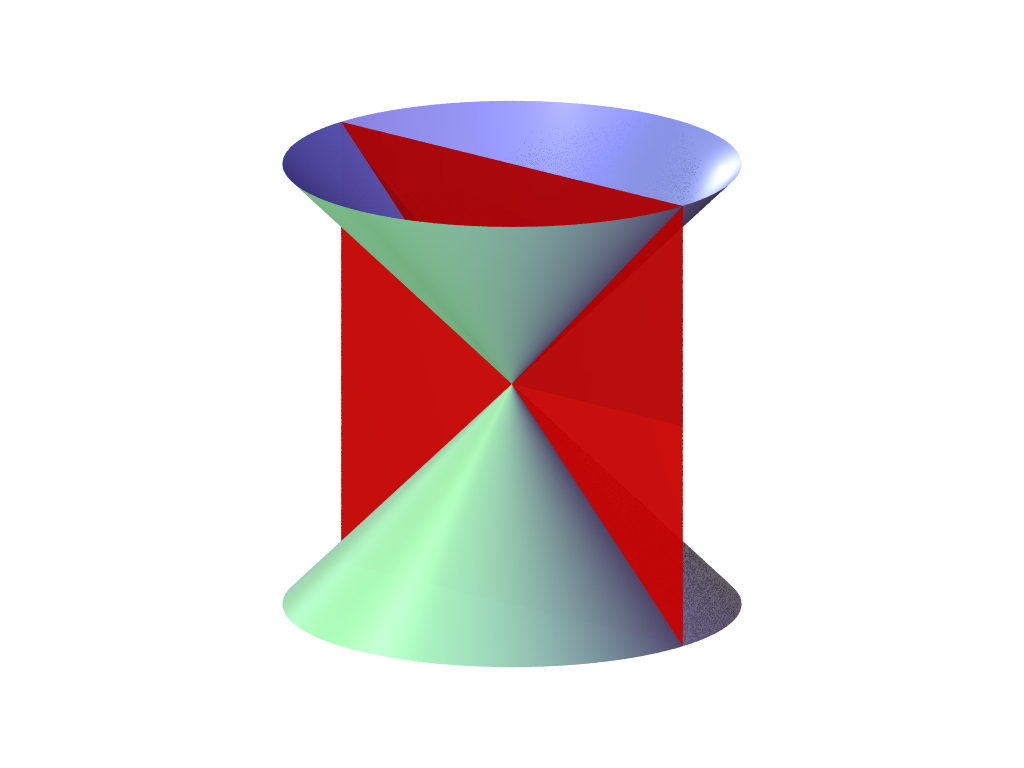}
\caption{ \label{Fig:coneplane} The cone $D_1$ and the plane $D_2$ are not splayed at the origin but satisfy $\Omega^1_V(\log D_1)+\Omega^1_V(\log D_2)=\Omega^1_V(\log D)$.}
\end{figure}
\qede\end{example}

\begin{remark}\label{rem:additivomega}
Example \ref{Ex:additivomega} shows that 
in general the first condition listed in~Theorem~\ref{Thm:Omegamodcrit} does 
not suffice to imply splayedness.
One can construct similar examples using Saito's theorem (Theorem~2.9 of \cite{Saito80}): 
let $D=\cup_{i=1}^mD_i$ at a point $p \in V$ be a divisor, where the $D_i$ are the 
irreducible components of $D$. Then Saito's theorem says that if each $D_i$ is normal, 
$D_i$ intersects $D_j$ ($i \neq j$) transversally outside a codimension $2$ set and the triple intersections $D_i \cap D_j \cap D_k$ have codimension $\geq 3$ (for $i \neq j \neq k$) then 
$$\Omega^1_{V,p}(\log D)=\sum_{i=1}^m \frac{dh_i}{h_i}\cO + \Omega_{V,p}^1.$$
This is easily seen to be equal to $\sum_{i=1}^m \Omega^1_{V,p}(\log D_i)$.
For non-splayed $D_i$
this gives examples where $\Ext^1_{\cO}(\Omega^1_V(\log D),\cO) \neq 0$.
\qede\end{remark}

\begin{remark}\label{rem:mathias}
Mathias Schulze pointed out to us that a situation similar to 
Theorem~\ref{Thm:Omegamodcrit} occurs by considering modules 
$\omega_D^\bullet$
of regular differential forms on $D$.
If $f$ defines a reduced divisor $D$, then the dual of 
the module $J'_f/(f)$ is the module $\mathcal R_D=\omega^0_D$ (\cite{GS}, Proposition~3.2). 
Dualizing the map $\iota\colon J'_{gh}/(gh) \hookrightarrow J'_g/(g) \oplus J'_h/(h)$ 
used in the proof of Theorem~\ref{Thm:Jpcrit} gives a natural inclusion
\[
\iota^\vee\colon \omega^0_{D_1} \oplus \omega^0_{D_2} \hookrightarrow \omega^0_D
\]
where $D=D_1\cup D_2$. If $D_1$ and $D_2$ are splayed, then $\iota$ is an
isomorphism, and so is~$\iota^\vee$. If $D$ is free, then the converse holds:
if $D$ is free and $\iota^\vee$ is an isomorphism, then $\omega^0_D$, 
resp.~$\omega^0_{D_1}$, $\omega^0_{D_2}$ are reflexive
(\cite{GS}, Corollary~3.5), so dualizing
$\iota^\vee$ shows that $\iota$ is also an isomorphism; it follows that
$D$ is splayed, by Theorem~\ref{Thm:Jpcrit}.
One advantage of this observation over Theorem~\ref{Thm:Omegamodcrit}
is that the module $\omega^0_D$ only depends on $D$, not on the embedding of 
$D$ in a nonsingular variety $V$.
\qede\end{remark}

\subsection{Global considerations, and Chern classes}
The global version of Theorem~\ref{Thm:Dermodcrit} is the following immediate
consequence at the level of sheaves of derivations.

\begin{theorem}\label{Thm:Dercrit}
Let $D_1$, $D_2$ be reduced divisors of $V$, without common components, and
let $D=D_1\cup D_2$. Then there is a natural monomorphism of sheaves
\[
\frac{\Der_V}{\Der_V(-\log D)} \hookrightarrow
\frac{\Der_V}{\Der_V(-\log D_1)}\oplus \frac{\Der_V}{\Der_V(-\log D_2)}\quad,
\]
and $D_1$, $D_2$ are splayed if and only if this monomorphism is an 
isomorphism.
\end{theorem}

We can also globalize the splayedness module and introduce a `splayedness
sheaf'
\[
\Spla_V(D_1,D_2):=\frac {\Der_V}{\Der_V(-\log D_1)+\Der_V(-\log D_2)}\quad.
\]
so that $D_1$ and $D_2$ are splayed if and only if $\Spla_V(D_1,D_2)$
vanishes.

\begin{example}\label{ex:curvecase1}
For reduced curves on surfaces, splayedness is equivalent to transversality at
nonsingular points; this is easily deduced from the definition of splayedness.
Namely, if $C_1$ and $C_2$ are reduced curves on a surface $S$, then $C_1$ 
and $C_2$ are splayed at a point $p$ if one can find coordinates $(x,y)$ at $p$ 
such that $C_1$ is locally defined by $g(x,0)$ and $C_2$ by $h(0,y)$. Since 
$C_1$ and $C_2$ are reduced, this is only possible when $g(x,0)$ and $h(0,y)$ 
are of the form $ux$ and $vy$ for some units $u,v \in \cO_{S,p}$.
Putting together this remark and
the results proved so far, we see that if $C_1$, $C_2$ are reduced curves 
on a compact surface $S$, and we let $C=C_1\cup C_2$, then
the following are equivalent:
\begin{itemize}
\item The natural inclusion $\Der_S(-\log C_1)+\Der_S(-\log C_2)
\hookrightarrow \Der_S$ is an equality.
\item The natural inclusion $\Omega^1_S(\log C_1)+\Omega^1_S(\log C_2)
\hookrightarrow \Omega^1_S(\log C)$ is an equality.
\item $C_1$ and $C_2$ are splayed.
\item $C_1$ and $C_2$ meet transversally at nonsingular points.
\end{itemize}
One more item will be added to this list in \S\ref{sec:CSMsplayed}:
\begin{itemize}
\item $\csm(S\smallsetminus C_1)\cdot \csm(S\smallsetminus C_2)
=c(TS)\cap \csm(S\smallsetminus C)\quad,$
\end{itemize}
see Example~\ref{sec:curvecase2}. This will use {\em Chern-Schwartz-MacPherson\/}
classes; Chern classes are our next concern.
\qede\end{example}

In the next sections we will be interested in the behavior of splayedness
{\em vis-a-vis\/} a conjectural statement on Chern classes of sheaves of
logarithmic derivations. Recall (\cite{85k:14004}, \S15.1 and B.8.3) that on 
nonsingular varieties one can define Chern classes for any coherent sheaf, 
compatibly with the splitting principle: the key fact is that every 
coherent sheaf on a nonsingular variety admits a finite resolution by locally 
free sheaves. Thus, for any hypersurface $D$ on a nonsingular variety $V$
we may consider the class
\[
c(\Der_V(-\log D))
\]
in the Chow ring of $V$, or its counterpart $c(\Der_V(-\log D))\cap [V]$
in the Chow group of $V$. (The reader will lose very little by considering
these classes in the cohomology, resp.~homology of $V$.)

In terms of Chern classes, Theorem~\ref{Thm:Dercrit} has the following immediate
consequence:

\begin{corol}\label{cor:Chern}
Let $D_1$, $D_2$ be reduced divisors of $V$, without common components, and
let $D=D_1\cup D_2$. If $D_1$ and $D_2$ are splayed, then
\[
c(\Der_V(-\log D)) = \frac{c(\Der_V(-\log D_1))\cdot c(\Der_V(-\log D_2))}{c(\Der V)}
\]
in the Chow ring of $V$.
\end{corol}

\begin{remark}
This Chern class statement only uses the `easy' implication in the criterion for
splayedness of Theorem~\ref{Thm:Dercrit}. It does not seem likely that
splayedness can be precisely detected by a Chern class computation; of course,
Corollary~\ref{cor:Chern} may be used to prove that two divisors are {\em not\/}
splayed.
\qede\end{remark}

\begin{proof}
If $D_1$ and $D_2$ are splayed, then by Theorem~\ref{Thm:Dercrit} we have
an isomorphism
\[
\frac{\Der_V}{\Der_V(-\log D)} \cong
\frac{\Der_V}{\Der_V(-\log D_1)}\oplus \frac{\Der_V}{\Der_V(-\log D_2)}
\]
of coherent sheaves, and taking Chern classes we get
\[
\frac{c(\Der_V)}{c(\Der_V(-\log D))} =
\frac{c(\Der_V)}{c(\Der_V(-\log D_1))}\cdot \frac{c(\Der_V)}{c(\Der_V(-\log D_2))}
\]
in the Chow ring of $V$. The stated equality follows at once.
\end{proof}

\begin{example}\label{exa:SNCexample}
As an illustration of Corollary~\ref{cor:Chern}, consider a divisor $D$ with 
normal crossings and nonsingular components $D_i$. First we note that by
sequence~\eqref{seq:globalder}, if $D=D_1$ is nonsingular, then
\[
c(\Der_V(-\log D))=\frac{c(\Der_V)}{1+D}\quad,
\]
where $1+D_1$ is the common notation for $c(\cO_V(D))=1+c_1(\cO_V(D))$.
Indeed, $JD=\emptyset$ as $D$ is nonsingular, so $\cI_{JD,D}(D)=
\cO_D(D)$, and twisting the standard exact sequence for $\cO_D$ gives
the sequence
\[
\xymatrix{
0 \ar[r] &
\cO_V \ar[r] &
\cO_V(D) \ar[r] &
\cO_D(D) \ar[r] &
0\quad,
}
\]
showing that $c(\cO_D(D))=c(\cO_V(D))/c(\cO_V)=1+D$.

Now the claim is that if $D=D_1\cup \cdots \cup D_r$ is a divisor with normal
crossings and nonsingular components, then
\[
c(\Der_V(-\log D))=\frac{c(\Der_V)}{(1+D_1)\cdots (1+D_r)}\quad.
\]
This formula is in fact well-known: it may be obtained by computing explicitly 
the ideal of the singularity subscheme $JD$ and taking Chern classes of the 
corresponding sequence~\eqref{seq:globalder}. The point we want to make
is that this formula follows immediately from Corollary~\ref{cor:Chern}, 
without any explicit computation of ideals. Indeed, by the normal 
crossings condition, $D_1\cup\cdots\cup D_{i-1}$ and $D_i$ are splayed for 
all $i>1$; the formula holds for $r=1$ by the 
explicit computation given above; and for $r>1$ and induction, we have
\begin{align*}
c(\Der_V(-\log D)) &=\left. \frac{c(\Der_V)}{(1+D_1)\cdots (1+D_{r-1})}
\cdot \frac{c(\Der_V)}{(1+D_r)}\right/ c(\Der_V) \\
&=\frac{c(\Der_V)}{(1+D_1)\cdots (1+D_r)}
\end{align*}
as claimed.
\qede\end{example}


\section{Chern-Schwartz-MacPherson classes for splayed divisors
and subvarieties}\label{sec:CSMsplayed}

\subsection{CSM classes}\label{CSMintro}
There is a theory of Chern classes for possibly singular, possibly noncomplete
varieties, normalized so that the class for a nonsingular compact variety~$V$
equals the total Chern class $c(TV)\cap [V]$ of the tangent bundle of $V$,
in the Chow group (or homology), and satisfying a strict functoriality requirement. 
This theory was developed by R.~MacPherson (\cite{MR0361141}); 
\S19.1.7 in~\cite{85k:14004} contains an efficient summary of MacPherson's 
definition. These `Chern classes' were found to agree with a notion defined 
in remarkable earlier work by M.-H.~Schwartz (\cite{MR35:3707, MR32:1727}) 
aimed at extending
the Poincar\'e-Hopf theorem to singular varieties; they are usually called
{\em Chern-Schwartz-MacPherson\/} ($\csm$) classes. A $\csm$ class
$\csm(\varphi)$ is defined for every constructible function $\varphi$ on a 
variety; the key functoriality of these classes prescribes that if $f: V\to W$ is
a proper morphism, and $\varphi$ is a constructible function on $V$, then
$f_* \csm(\varphi) = \csm (f_* \varphi)$. Here, $f_*\varphi$ is defined by 
taking topological Euler characteristics of fibers. This covariance property
also determines the theory uniquely by resolution of singularity and the
normalization property mentioned above.

We mention here two immediate consequences of functoriality that are
useful in computations. A locally closed subset $U$ of a variety $V$ determines 
a $\csm$ class $\csm(U)$ in the Chow group of $V$: this is the $\csm$ class 
of the function $\one_U$ which takes the value~$1$ on~$U$ and $0$ on its 
complement. 
\begin{itemize}
\item If $V$ is compact and $U\subseteq V$ is locally closed, then the degree 
$\int \csm(U)$ equals the topological Euler characteristic of $U$. 
Thus, $\csm$ classes satisfy a generalized version of the Poincar\'e-Hopf
theorem,
and may be viewed as a direct generalization of the topological Euler
characteristic.
\item Like the Euler characteristic, $\csm$ classes satisfy an inclusion-exclusion 
principle: if $U_1$, $U_2$ are locally closed in $V$, then 
\[
\csm(U_1\cup U_2) = \csm(U_1) + \csm(U_2) - \csm(U_1\cap U_2)\quad.
\]
\end{itemize}

Chern classes of bundles of logarithmic derivations along a divisor with
simple normal crossings may be used to provide a definition of $\csm$ classes.
This approach is adopted in~\cite{MR2282409},~\cite{MR2209219}; 
a short summary may be found in~\S3.1 of~\cite{MR2448279}.
In this section we explore the role of splayedness in more refined (and 
still conjectural in part) relations between $\csm$ classes and Chern
classes of sheaves of logarithmic derivations.

\subsection{CSM classes of hypersurface complements}
We now consider $\csm$ classes of hypersurface complements. 
As in \S\ref{sec:splay} we will assume that $V$ is a nonsingular complex 
projective variety. In previous work, the first-named author has proposed 
a formula relating the $\csm$ class of
the complement $U=V\smallsetminus D$ of a divisor in a nonsingular
variety $V$ with the Chern class $c(\Der_V(-\log D))\cap [V]$ of the
corresponding sheaf of logarithmic derivations.

\begin{example}
Using the inclusion-exclusion formula for $\csm$ classes given above, it is 
straightforward to compute the $\csm$ class of the
complement of a divisor with simple normal crossings $D=D_1\cup \cdots
\cup D_r$ (cf.~e.g., Theorem~1 in \cite{MR2001d:14008} or 
Proposition 15.3 in \cite{MR1893006}), and verify that in this case
\begin{equation}\label{eq:conj}
\csm(V\smallsetminus D) = c(\Der(-\log D))\cap [V]
\end{equation}
by direct comparison with the class computed in Example~\ref{exa:SNCexample}.
\qede\end{example}

It is natural to inquire whether equality~\eqref{eq:conj} holds for less special
divisors. It has been verified for free hyperplane arrangements
(\cite{hyparr}) and more generally for free hypersurface arrangements that are
locally analytically isomorphic to hyperplane arrangements (\cite{hypersarr}).
Xia Liao has verified it for locally quasi-homogeneous curves on a nonsingular 
surface (\cite{Liao1}), and he has recently proved that the formula holds
`numerically' for all free and locally quasi-homogeneous hypersurfaces of 
projective varieties (\cite{Liao2}).

On the other hand, the formula is known {\em not\/} to hold in general: for
example, Liao proves that the formula does not hold for curves on surfaces
with singularities at which the Milnor and Tyurina numbers do not coincide.

\subsection{Enter splayedness}
Rather than focusing on the verification of \eqref{eq:conj} for cases not 
already covered by these results, we aim here to consider a consequence
of \eqref{eq:conj} in situations where it does hold. In Corollary~\ref{cor:Chern}
we have verified that if $D_1$ and $D_2$ are splayed, then there is a simple 
relation between the Chern classes of the sheaves of logarithmic derivations
determined by $D_1$, 
$D_2$, and $D=D_1\cup D_2$. If we assume for a moment that 
\eqref{eq:conj} holds for these hypersurfaces, we obtain a non-trivial 
relation between the $\csm$ classes of the corresponding complements. 
This relation can then be probed with independent tools, and cases in
which it is found to hold may be viewed as a consistency check for 
a general principle linking $\csm$ classes of hypersurface complements
and Chern classes of logarithmic derivations, of which \eqref{eq:conj}
is a manifestation.

\begin{prop}
Let $D_1$, $D_2$ be reduced divisors of $V$, without common components, 
and let $D=D_1\cup D_2$. Assume that $D_1$ and $D_2$ are splayed, and
that \eqref{eq:conj} holds for $D_1$, $D_2$, and $D$. Let $U$, resp.~$U_1$,
$U_2$ be the complement of $D$, resp.~$D_1$, $D_2$. Then
\begin{equation}\label{eq:template}
c(TV)\cap \csm(U) = \csm(U_1) \cdot \csm(U_2)
\end{equation}
in the Chow group of $V$.
\end{prop}

\begin{proof}
This is a direct consequence of Corollary~\ref{cor:Chern}, under the assumption
that \eqref{eq:conj} holds for all hypersurfaces, and noting that $\Der_V$ is the
sheaf of sections of the tangent bundle.
\end{proof}

In the rest of this section we are going to verify \eqref{eq:template} in a few
template situations, independently of the conjectural formula \eqref{eq:conj}. 
In fact, it can be shown that \eqref{eq:template} {\em holds
for all splayed divisors.\/} However, the proof of this fact is rather technical, 
and we hope to return to it in later work. The situations we will consider in this 
paper can be appreciated with a minimum of machinery.

\subsection{Products}\label{sec:products}
Formula~\eqref{eq:template} can in fact be stated for splayed 
{\em subvarieties,\/} or in fact arbitrary closed subsets, rather than just 
divisors. Let $V_1$, $V_2$ be nonsingular varieties, let $X_1\subseteq V_1$, 
$X_2\subseteq V_2$ be closed subsets; then $D_1=X_1\times V_2$
and $D_2=V_1\times X_2$ may be considered to be splayed at all points
of the intersection $X_1\times X_2$. In general, two closed subsets $D_1$, 
$D_2$ of a nonsingular
variety $V$ are splayed at $p\in D_1\cap D_2$ if locally analytically at $p$,
$D_1$, $D_2$ and $V$ admit the product structure detailed above.

If this local analytic description holds globally, then \eqref{eq:template} 
is a straightforward consequence of known properties of $\csm$ classes.

\begin{prop}\label{pro:products}
Formula \eqref{eq:template} holds for $D_1=X_1\times V_2$ and 
$D_2=V_1\times X_2$ in $V=V_1\times V_2$.
\end{prop}

\begin{proof}
Let $U_1=V\smallsetminus D_1$, $U_2=V\smallsetminus D_2$, $U=
V\smallsetminus D$. In the special situation of this proposition,
\begin{align*}
U &=(V\smallsetminus (D_1\cup D_2))
=(V\smallsetminus D_1)\cap (V\smallsetminus D_2)
=((V_1\smallsetminus X_1)\times V_2) \cap (V_1\times (V_2\smallsetminus X_2))\\
&=(V_1\smallsetminus X_1) \times (V_2\smallsetminus X_2)
\end{align*}
Now we invoke a product formula for $\csm$ classes (\cite{MR1158750}, 
\cite{MR2209219}): by Th\'eor\`eme~4.1 in \cite{MR2209219},
\[
\csm((V_1\smallsetminus X_1) \times (V_2\smallsetminus X_2))
=\csm(V_1\smallsetminus X_1) \otimes \csm(V_2\smallsetminus X_2)\quad,
\] 
where $\otimes$ denotes the natural morphism $A_*(V_1)\otimes A_*(V_2)
\to A_*(V_1\times V_2)=A_*(V)$ sending $\alpha_1\otimes \alpha_2$
to $(\pi_1^* \alpha_1)\cdot (\pi_2^*\alpha_2)$, where $\pi_1$, resp.~$\pi_2$ 
is the projection from $V_1\times V_2$ to the first, resp.~second factor. 
As $c(TV)=\pi_1^* c(TV_1)\cap \pi_2^* c(TV_2)$,
\[ 
c(TV)\cap \csm(U)=(\pi_2^*c(TV_2)\cap \pi_1^* \csm(V_1\smallsetminus X_1))
\cdot (\pi_1^*c(TV_1)\cap \pi_1^* \csm(V_2\smallsetminus X_2))\quad.
\]
Finally we note that by Theorem~2.2 in~\cite{MR1695362}
\[
\pi_2^*c(TV_2)\cap \pi_1^* \csm(V_1\smallsetminus X_1) = 
\csm((V_1\smallsetminus X_1)\times V_2) = \csm(V\smallsetminus D_1)
\]
and similarly for the other factor, concluding the proof.
\end{proof}

\begin{remark}\label{re:gentoclosed}
The fact that \eqref{eq:template} generalizes to complements of more general
closed subsets is not surprising, as it is a formal consequence of the formulas 
for complements of divisors.
\qede\end{remark}

\begin{remark}
It is natural to expect that one could now deduce the validity of \eqref{eq:template}
for arbitrary splayed divisors from Proposition~\ref{pro:products} and some
mechanism obtaining intersection-theoretic identities from local analytic data.
J\"org Sch\"urmann informs us that his Verdier-Riemann-Roch theorem
(\cite{Schup}) can be
used for this purpose;
our proof of \eqref{eq:template}
for all splayed divisors (presented elsewhere) relies on different tools.
\qede\end{remark}

\subsection{Joins}\label{sec:joins}
The classical construction of `joins' in projective space (see e.g.,~\cite{MR1182558})
gives another class of examples of splayed divisors and subvarieties for which
\eqref{eq:template} can be reduced easily to known results. We deal directly with the
case of subvarieties, cf.~Remark~\ref{re:gentoclosed}.
We recall that 
the {\em join\/} of two disjoint subvarieties of projective space is the union of the
lines incident to both. For example, the ordinary cone with vertex a point $p$ and
directrix a subvariety $X$ is the join $J(p,X)$.

Let $X_1\subseteq \Pbb^{m-1}$, $X_2\subseteq \Pbb^{n-1}$ be nonempty
subvarieties, and view $\Pbb^{m-1}$, $\Pbb^{n-1}$ as disjoint subspaces of
$V=\Pbb^{m+n-1}$. Let $D_1=J(X_1,\Pbb^n)$, resp.~$D_2=J(\Pbb^m,X_2)$ 
be the corresponding joins. The intersection $J(X_1,X_2)$ of $D_1$ 
and $D_2$ is the union of the set of lines in $\Pbb^{m+n-1}$ connecting points 
of $X_1$ to points of $X_2$. The subsets $D_1$, $D_2$ are evidently splayed 
along their intersection $J(X_1,X_2)$; but note that $V$ is not a product, so 
Proposition~\ref{pro:products} does not apply in this case.

\begin{prop}\label{pro:joins}
Formula~\eqref{eq:template} holds for $D_1$, $D_2$, $V=\Pbb^{m+n-1}$ as above.
\end{prop}

\begin{proof}
We have to compare
\begin{multline*}
\csm(V\smallsetminus D_1)\cdot \csm(V\smallsetminus D_2)
=(\csm(V)-\csm(D_1))\cdot (\csm(V)-\csm(D_2)) \\
=\csm(V)\cdot \csm(V) - \csm(V)\cdot (\csm(D_1)+\csm(D_2)) + \csm(D_1)\cdot \csm(D_2)
\end{multline*}
and
\begin{multline*}
c(TV)\cap \csm(V\smallsetminus (D_1\cup D_2))
=c(TV)\cap (\csm(V)-(\csm(D_1)+\csm(D_2))+\csm(D_1\cap D_2)) \\
=c(TV)\cap \csm(V) - c(TV)\cap (\csm(D_1)+\csm(D_2)) + 
c(TV)\cap \csm(D_1\cap D_2)\,.
\end{multline*}
By the basic normalization of $\csm$ classes (cf.~\S\ref{CSMintro}), capping
with $c(TV)$ is the same as taking the intersection product with $\csm(V)$.
Since $D_1\cap D_2=J(X_1,X_2)$, we are reduced to verifying that
\[
\csm(D_1)\cdot \csm(D_2) = c(TV)\cap \csm(J(X_1\cap X_2))\quad.
\]
The classes $\csm(X_1)\in A_*\Pbb^{m-1}$, resp.~$\csm(X_2)\in A_*\Pbb^{n-1}$ 
may be written as polynomials $\alpha$, resp.~$\beta$ of degree $< m$,
resp.~$n$ in the hyperplane class in these subspaces. Denoting by $H$ the 
hyperplane class in $V=\Pbb^{m+n-1}$,
we obtain formulas for $\csm(D_1)$, $\csm(D_2)$ in $A_*V$ by applying
Example~6.1 in \cite{MR2782886} (and noting that $H^{n+m}=0$ 
in~$A_*\Pbb^{m+n-1}$):
\begin{align*}
\csm(D_1) &= (1+H)^n (\alpha(H)+H^m)\cap [\Pbb^{m+n-1}]\\
\csm(D_2) &= (1+H)^m (\beta(H)+H^n)\cap [\Pbb^{m+n-1}]\quad.
\end{align*}
Therefore
\begin{align*}
\csm(D_1)\cdot \csm(D_2)
&=(1+H)^{m+n}(\alpha(H)+H^m)\cdot (\beta(H)+H^{m+n})\cap [\Pbb^{m+n-1}] \\
&=c(T \Pbb^{m+n-1})\cap 
\left((\alpha(H)+H^m)\cdot (\beta(H)+H^{m+n})\cap [\Pbb^{m+n-1}]\right)
\end{align*}
This equals $c(TV)\cap \csm(J(X_1,X_2))$ by Theorem~3.13 in~\cite{MR2782886},
again noting that $H^{m+n}=0$ in $A_*\Pbb^{m+n-1}$.
\end{proof}

\subsection{Splayed curves}\label{sec:curvecase2}
Finally, we deal with the case of splayed curves on surfaces. In this case,
\eqref{eq:template}
is a {\em characterization\/} of splayedness (cf.~Example~\ref{ex:curvecase1}).

Let $C_1$ and $C_2$ be reduced curves on a nonsingular compact
surface $S$, and let $C=C_1\cup C_2$. 

\begin{prop}\label{pro:curves}
$C_1$ and $C_2$ are splayed on $S$ if and only if \eqref{eq:template} holds,
that is, $c(TS)\cap \csm(S\smallsetminus C) = \csm(S\smallsetminus C_1)\cdot
\csm(S\smallsetminus C_2)$.
\end{prop}

\begin{proof}
We have
\[
\csm(S\smallsetminus C_i) = c(TS)\cap [S] - [C_i] - \chi_i
\]
for $i=1,2$, where $\chi_i$ is a class in dimension~$0$ (whose degree is
the topological Euler characteristic of $C_i$). By inclusion-exclusion,
\begin{align*}
\csm(C) &= \csm(C_1)+\csm(C_2)-\csm(C_1\cap C_2) \\
&=[C_1]+[C_2]+\chi_1+\chi_2-[C_1\cap C_2]\quad,
\end{align*}
and hence
\[
\csm(S\smallsetminus C) = c(TS)\cap [S] - [C_1] - [C_2] - \chi_1 - \chi_2 + [C_1\cap C_2]
\quad.
\]
It then follows at once that
\[
\csm(S\smallsetminus C_1) \cdot \csm(S\smallsetminus C_2) 
-c(TS)\cap \csm(S\smallsetminus C) = C_1\cdot C_2 - [C_1\cap C_2]\quad.
\] 
Therefore, in this case \eqref{eq:template} is verified if and only if $C_1\cdot C_2 =
[C_1\cap C_2]$, that is, if and only if $C_1$ and $C_2$ meet transversally at 
nonsingular points.
As we have recalled in \S\ref{sec:splay}, this condition is equivalent to the requirement
that $C_1$ and $C_2$ are splayed, verifying~\eqref{eq:template} in this case and 
proving that this identity {\em characterizes\/} splayedness for curves
on surfaces.
\end{proof}

\begin{remark}\label{re:transversality}
In the proof of both Propositions~\ref{pro:joins} and~\ref{pro:curves} we have used
the fact that~\eqref{eq:template} is equivalent to the identity
\begin{equation}\label{eq:transversality}
\csm(D_1)\cdot \csm(D_2) = c(TV)\cap \csm(D_1\cap D_2) \quad.
\end{equation}
If $D_1$ and $D_2$ are nonsingular subvarieties of $V$ intersecting properly
and transversally (so that $D_1\cap D_2$ is nonsingular, of the expected 
dimension), then \eqref{eq:transversality} is precisely the expected relation
between the Chern classes of the tangent bundles of $D_1$, $D_2$, and $D_1
\cap D_2$. The results of the previous sections verify~\eqref{eq:transversality}
for several classes of splayed subvarieties (nonsingular or otherwise), and 
we will prove elsewhere that in fact \eqref{eq:transversality} holds in general
for the intersection of two splayed subvarieties.
This reinforces the point of view taken by the second author in~\cite{Faber12},
to the effect that splayedness is an appropriate generalization of transversality
for possibly singular varieties. For another situation in which the formula
holds when one (but not both) of the hypersurfaces is allowed to be singular,
see Theorem~3.1 in~\cite{ccdc}.
\qede\end{remark}


\end{document}